\documentclass{amsart}
\usepackage{color}
\usepackage[left=3.2cm, right=3.2cm, top=3cm, bottom=3cm]{geometry}
\usepackage{amsmath, amssymb, amsfonts, amstext, amsthm, textcomp, comment,bbm, mathtools}
\usepackage{nccmath,enumitem, bbm, comment, lineno, braket}

  \usepackage[normalem]{ulem}
\usepackage{chngcntr}
\usepackage{enumitem}
\usepackage{float}

\newtheorem{theorem}{Theorem}[section]
\newtheorem{lemma}[theorem]{Lemma}
\newtheorem{corollary}[theorem]{Corollary}
\newtheorem{proposition}[theorem]{Proposition}

\theoremstyle{definition}
\newtheorem{definition}[theorem]{Definition}
\newtheorem{example}[theorem]{Example}
\newtheorem{claim}{Claim}
\newtheorem{conj}{Conjecture}

\theoremstyle{remark}
\newtheorem{remark}[theorem]{Remark}

\numberwithin{equation}{section}

\newcommand{\N}{\mathbb{N}}

\newcommand{\R}{\mathbb{R}}

\newcommand{\calB}{\mathcal{B}}

\DeclareMathOperator{\rank}{rank}

\usepackage[hidelinks]{hyperref}

\newcommand{\tr}{\operatorname{tr}} %

\begin{document}

\title{Riesz* Homomorphisms on the copositive Cone}

\author{Pavankumar Raickwade}
\address{Department of Mathematics, IIT Madras, India}
\curraddr{}
\email{prraickwade@gmail.com}
\thanks{} 

\author{K. C. Sivakumar}
\address{Department of Mathematics, IIT Madras, India}
\curraddr{}
\email{kcskumar@iitm.ac.in}
\thanks{} 



\date{}

\dedicatory{}

\begin{abstract}
    For a cone $K\subseteq~\R^n$, a real symmetric matrix $A$ is called $K$-copositive if $x^\top A x\geq 0$ for every $x\in K.$ This class of matrices plays a central role in copositive optimization and linear complementarity problems. However, a complete characterization of linear maps that preserve the $K$-copositive cone is unknown, even for $K:=\R^n_+$.
    
    In this paper, we develop a new approach to copositivity preservers that uses only order-theoretic arguments. We consider a smaller class of copositivity preservers, called Riesz* homomorphisms, and develop a general technique to deduce the structure of these preservers directly from a representation theorem of Riesz* homomorphisms on $S_n$. Following are the main outcomes of this paper:
    \begin{enumerate}
        \item 
            We obtain a representation theorem for Riesz* homomorphisms on the partially ordered vector space of all real symmetric matrices endowed with the cone of all $K$-copositive matrices.
        \item 
            As a corollary of our representation theorem, we recover the main results of [Shitov, Proc. Amer. Math. Soc., 2021] and [Gowda et al., Linear Alg. Appl., 2013], providing a unified framework for studying cone automorphisms.
        \item 
            We introduce the notion of a $(K_1,K_2)$-unisigned matrix $P\in M_{m\times n}$, defined by the algebraic condition $P[K_1]\subseteq K_2\cup (-K_2)$, for cones $K_1\subseteq \R^n$ and $K_2\subseteq \R^m$. We also provide a characterization of such matrices.
        \item 
            We prove that a linear map of the standard form ($A\mapsto P^\top AP$; for $P\in M_{m\times n}$) preserves copositivity if and only if $P$ is $(K_2,K_1)$-unisigned, correcting a recent characterization of such maps preserving the $\R^n_+$-copositivity.
        
    \end{enumerate} 
\end{abstract}

\maketitle


\section{Introduction}\label{intro}

Let $M_{m\times n}$ denote the set of all $m\times n$ real matrices, and we let $M_n:=M_{n\times n}$. Let $S_n\subseteq M_n$ denote the set of all $n\times n$ real symmetric matrices. We endow $S_n$ with the inner product $\langle A, B\rangle =\tr(AB)$. Given a closed and generating cone $K\subseteq \R^n$, we consider the corresponding cone 
\[
    COP(K):=\Set{A\in S_n \ |\ \forall x\in K\colon x^\top Ax\geq 0}
\] 
of all {\it $K$-copositive matrices}. A linear map $T\colon S_n\to S_n$ is called an {\it into preserver of $K$-copositive matrices} if $K$-copositivity of $A$ implies $K$-copositivity of $T(A)$. Into preservers of $\R^n_+$-copositive matrices have been studied in \cite{shitov2019linearmappingspreservingcopositive, Furtado27072021}, whereas the group of automorphisms of the $K$-copositive matrices $Aut(COP(K))$ have been studied in \cite{GOWDA20133862}. In this context, we note the following.
\begin{enumerate}[label=\upshape(\roman*)]
    \item 
        The full characterization of into preservers of $K$-copositive matrices is still unknown, even for $K:=\R^n_+$.
    \item 
        In \cite{Furtado27072021}, the following claim and conjecture were proposed.
        \begin{claim}[{\cite[Theorem 2.2]{Furtado27072021}}]\label{claim1}
            A linear map $T\colon S_n\to S_n$, given by $A\mapsto P^\top A P$, is an into preserver of $\R^n_+$-copositivity if and only if $P$ is entry-wise nonnegative.
        \end{claim}
        \begin{conj}\label{conj:1}
            For a linear map $T\colon S_n\to S_n$, $T\in Aut(COP(\R^n_+))$ if and only if there exists $P\in Aut(\R^n_+)$ such that $T(A)=P^\top A P$ for every $A\in S_n$.
        \end{conj}
   
    \item 
        Conjecture~\ref{conj:1} was settled in the affirmative in \cite{shitov2019linearmappingspreservingcopositive}.
    \item 
         More generally, for a closed and generating cone $K\subseteq \R^n$, the question of characterizing $Aut(COP(K))$ had already been investigated in \cite{GOWDA20133862}, where a solution was provided in terms of dual cones.
\end{enumerate}
Riesz spaces or vector lattices are special examples of partially ordered vector spaces where the suprema between any two vectors exist. In such a setting, Riesz homomorphisms are those maps that preserve suprema of any two elements. Recently, several generalizations of Riesz homomorphisms have been proposed and investigated over partially ordered vector spaces \cite{Haandel+1993, The_complete_Riesz_hom_preliminary}. In this context, it appears that the notion of Riesz* homomorphism is the most natural generalization of Riesz homomorphisms to partially ordered vector spaces. For more recent investigations on Riesz* homomorphisms, refer to \cite{Florian, Florian-Valentin-Anke-Janko-Onno, Florian-Anke-Janko-Onno}. 
The space $(S_n, COP(K))$ is generally not a Riesz space, and in this case the set of all positive operators coincide with the  set of all into preservers of $K$-copositive matrices. Since Riesz* homomorphisms are always positive, this naturally motivates the study of Riesz* homomorphisms on $(S_n, COP(K))$ as a means of further advancing the theory of into preservers of $K$-copositive matrices.  

The present note serves the following purposes. Let  $K_1\subseteq \R^n$ and $K_2\subseteq \R^m$ be closed and generating cones.
\begin{enumerate}[label=\upshape(\roman*)]
  \item   
        We introduce a new class of matrices, called $(K_1,K_2)$-unisigned matrices (Definition~\ref{def:unisigned}), for which we obtain a characterization result in Theorem~\ref{thm::prop(U)_char}. 
    \item 
        In Theorem~\ref{thm::cp-rank-one-non-increasing}, we first investigate maps $T\colon S_n\to S_m$ that satisfy
        \[
            T\left[\Set{uu^\top \ | \ u\in K_1}\right]\subseteq \Set{uu^\top \ | \ u\in K_2},
        \]
        using which, in Theorem~\ref{thm::Riesz*hom_on_X_m}, we obtain a representation theorem for the Riesz* homomorphisms $T\colon (S_n, COP(K_1))\to (S_n, COP(K_2))$.
    \item 
       As a corollary of our representation theorem for Riesz* homomorphisms on $(S^n, COP(K))$, as stated in Corollary~\ref{cor::to Riesz*}, we recover both Conjecture~\ref{conj:1}, which was proved in \cite{shitov2019linearmappingspreservingcopositive}, and the main result of \cite{GOWDA20133862}.
    \item 
        We have identified a gap in the proof of Claim~\ref{claim1} that invalidates the conclusion of the claim. We provide a simple counter-example (Example~\ref{count-example}). In Theorem~\ref{thm::copositive_characterizartion}, we propose and prove  the following claim as a corrected and more generalized version of Claim~\ref{claim1}.
        \begin{claim}\label{claim2}
             Let $T\colon S_n\to S_m$, be defined by $T(A)=P^\top A P,~A\in S_n.$ {Then $T$} maps $K_1$-copositive matrices to $K_2$-copositive matrices if and only if $P$ is $(K_2,K_1)$-unisigned.
        \end{claim}

\end{enumerate}

 To the best of our knowledge, the perspective of using Riesz* homomorphisms in the study of into preservers of $K$-copositive matrices has not been explored previously. The results presented in this article provide strong evidence that this approach is both natural and fruitful, opening new avenues for the study of linear preservers in ordered matrix spaces.

In Section \ref{prelims}, we recall preliminaries about partially ordered vector spaces, positive operators and Riesz* homomorphisms. In Section~\ref{Sec::uni-signed}, we talk about $K$-unisigned matrices. Section \ref{Riesz*-hom-copositive} is devoted to Riesz* homomorphisms $T\colon (S_m, COP(K_1))\to (S_n, COP(K_2))$. In Section \ref{into-preservers}, we prove Claim~\ref{claim2}, amending the flaw in Claim~\ref{claim1}.

\section{Preliminaries}\label{prelims}

\subsection{Partially ordered vector spaces}\label{prelim:partially_ordered_v.s.}

Let $X$ be a real vector space. A partial order $\leq$ on $X$ is called a {\it vector space order} if the following compatibility relations hold: 
\begin{align*}
	\forall x,y,z \in X&\colon \ x \leq y \Longrightarrow x+z\leq y+z,\\
	\forall x,y\in X, \lambda\geq 0&\colon \ x \leq y \Longrightarrow \lambda x\leq \lambda y.
\end{align*}
In this case, we call $X$ a {\it partially ordered vector space}. A subset $K\subseteq X$ is called a {\it cone} if $K$ is closed under addition, $\lambda K\subseteq K$ for every $\lambda \geq 0$, and $K\cap (-K)=\{0\}$. 
If $\leq$ is a vector space order on $X$, then  
\begin{equation} 
	\label{eq:cone}
	X_+:=\{x \in X\mid x\geq 0\}
\end{equation} is a cone. On the other hand, if $K\subseteq X$ is a cone, then a vector space order on $X$ is given by \begin{equation}\label{cone_order_relation}
    x\leq y ~:\Longleftrightarrow~  y-x\in K,
\end{equation}
and for $X_+$ as in \eqref{eq:cone}  one obtains $X_+=K$. We denote a partially ordered vector space by $(X,K)$ in case we want to specify the corresponding cone. If $X$ has a norm $\|\cdot\|$ then $(X, K, \|\cdot \|)$ is called an {\it ordered normed space}. The books \cite{cones_and_duality, Anke_book} are excellent sources for information on cones and partially ordered vector spaces. A cone $K$ is said to be {\it generating} if $X=K - K$. Note that $K-K$ equals the span of $K$. Let $X^*$ denote the algebraic dual of $X$. If $K$ is generating, then the set $K^*:=\{f\in X^*\mid f(x) \geq 0,~\forall x\in K\}$ defines a cone in $X^*$, known as the {\it dual cone} of~$K$. The following result is well known.
\begin{proposition}[{\cite[Proposition 1.5.5]{Anke_book}}]\label{prop::xinKifff(x)>0}
   Let $(X,\|\cdot\|)$ be a normed space and $K\subseteq X$ a closed cone. Then $x\in K$ if and only if $f(x)\geq 0$ for every continuous $f\in K^*$.
\end{proposition}
\begin{definition}
    Let $(X,K)$ be a partially ordered vector space. An element $u\in K\setminus \{0\}$ is called an {\it extremal} if, for every $0\leq z\leq u$, there exists $\lambda \in [0, 1]$ such that $z=\lambda u$. The set of all extremals of $K$ is denoted by $\text{Ext}(K)$.
\end{definition}


Let $X$ be a partially ordered vector space. For $A\subseteq X$, let
\begin{enumerate}[label=\upshape(\roman*)]
    \item 
        $A^u:=\{x\in X\mid x\geq a,~\forall a\in A\}$ denote the set of all {\it upper bounds} of $A$,
    \item 
        $A^l:=\{x\in X\mid x\leq a,~\forall a\in A\}$ denote the set of all {\it lower bounds} of $A$,
    \item 
        $A^{ul}:=(A^u)^l$ and similarly $A^{lu}:=(A^l)^u$.
\end{enumerate}

\subsection{Operators on partially ordered vector spaces}

Let $(X,K_1)$ and $(Y,K_2)$ be partially ordered vector spaces. A linear map $T\colon X\rightarrow Y$ is called
\begin{itemize}
    \item [-]
        {\it positive} if $T[K_1]\subseteq K_2$;
    \item [-]
        {\it automorphism} if $T$ is invertible and both $T,T^{-1}$ are positive;
    \item [-]
        {\it Riesz* homomorphism} if, for every nonempty finite subset $F\subseteq X$, one has 
        \[
            T[F^{ul}]\subseteq (T[F])^{ul}.
        \]
\end{itemize}
The set of all positive maps on $X$ is denoted by $\pi(K_1, K_2)$, and the set of all automorphisms is denoted by $Aut(K_1, K_2)$. If $X=Y$ and $K_1=K_2=:K$, then we simply write $Aut(K)$ instead of $Aut(K, K).$ 

\begin{remark}\label{Rem::Riesz*-composition}
The following properties of Riesz* homomorphisms are useful in our discussion. We refer the reader to \cite[Section 2.3]{Anke_book}, for a detailed study of Riesz* homomorphisms. 
    \begin{enumerate}
        \item [(i)]
            By taking $F:=\{0\}$, it follows that every Riesz* homomorphism is positive.
        \item [(ii)] 
            The composition of two Riesz* homomorphisms is again a Riesz* homomorphism (see \cite[Proposition 2.3.21]{Anke_book}).        
    \end{enumerate}
\end{remark}
\begin{remark}
    Recall that the {\it least upper bound (supremum)} of any two elements $x,y\in X$ is the unique element $z\in \{x,y\}^u\cap \{x,y\}^{ul}$. Hence, if $T\colon X\to Y$ is a Riesz* homomorphism, then $T$ preserves the supremum of two elements, whenever it exists. The converse is not true, in general, i.e. if $T[\{x,y\}^{ul}]\subseteq \{T(x), T(y)\}^{ul}$ for any $x,y\in X$, then $T$ may not be a Riesz* homomorphism. A detailed discussion on this topic can be found in \cite{Florian}.
\end{remark}
We also use the following property of Riesz* homomorphisms.

\begin{proposition}\label{bipositive-is-Riesz*}
    Let $(X,K)$ be a partially ordered vector space and $T\colon X\to X$ be a bijective linear map. Then $T\in Aut(K)$ if and only if both $T$ and $T^{-1}$ are Riesz* homomorphisms. 
\end{proposition}
\begin{proof}
    Since every Riesz* homomorphism is positive, the sufficiency is straightforward. Let $T\in Aut(K)$ and $\emptyset \neq F\subseteq X$ be finite and $w\in F^{ul}$. We aim to show that $T(w)\in T[F]^{ul}.$ Let $z\in T[F]^{u}$. Let $x\in X$ be such that $z=T(x)$. Since $T^{-1}$ is positive, we get $x\in F^u$ and so $w\leq x$. Since $T$ is positive,  $T(w)\leq T(x)=z$, proving that $T(w)\in T[F]^{ul}$. Hence $T[F^{ul}]\subseteq T[F]^{ul}$, and so $T$ is a Riesz* homomorphism. A similar argument for $T^{-1}$ gives the result.
\end{proof}

We shall use the following characterization of Riesz* functionals. For a normed space $(X, \|\cdot \|)$, let $X'$ denote the set of all bounded linear functionals on $X$. We consider the operator norm on $X'$. For $S\subseteq X'$, let $\overline{S}$ denote the norm closure of $S$ in $X'$. If $X$ is finite dimensional, then $X^*=X'$. 

\begin{proposition}[{\cite[Propositions 1.5.11, 2.3.28 and 2.5.5]{Anke_book}}]\label{prop:Extremals-are-R*}
    Let $(X,K)$ be a finite-dimensional partially ordered vector space with $K$ being closed and having nonempty interior. Then $f\in K^*$ is a Riesz* homomorphism if and only if $f\in \overline{Ext(K^*)}$.
\end{proposition}

\subsection{The cone of $K$-completely positive matrices}\label{ex::COP}

For a closed and generating cone $K\subseteq \R^n$, the set 
    \[
       CP(K):=\Set{\sum_{i=1}^n  u_iu_i^\top \ |\  u_i\in K\cup (-K)}\subseteq S_n.
    \]
    forms a closed and generating cone (cf. \cite[Proposition 5]{GOWDA20133862}), known as the cone of {\it $K$-completely positive matrices}. We identify $S_n^*$ with itself using the inner product on $S_n$, given by $\langle A,B\rangle =\tr(AB)$. It is well known (cf. \cite[Proposition~7]{GOWDA20133862}) that 
    \[
        COP(K)^*=CP(K)\quad \text{and}\quad \text{Ext}(CP(K))=\{uu^\top \mid u\in K\}=:\Lambda.
    \]
    \begin{lemma}\label{cp-lemma}
        Let $K\subseteq \R^n$ be a closed cone. The following results hold.
        \begin{enumerate}[label=\upshape(\roman*)]
            \item \label{Lambda-characterization}
                For $u\in \R^n$, one has $uu^\top\in \Lambda$ if and only if $u\in K\cup (-K)$.
            \item \label{Lambda closed}
                $\Lambda$ is a closed set.
        \end{enumerate}
    \end{lemma}
    \begin{proof}
        \ref{Lambda-characterization}: Since $uu^\top = (-u)(-u)^\top$, if $u\in K\cup (-K)$, then by definition $uu^\top\in \Lambda$. Let  $u\in \R^n\setminus K$ be such that $uu^\top \in\Lambda$, so that there exists  $v\in K$ such that $uu^\top = vv^\top$. Since $K$ is closed, by Proposition~\ref{prop::xinKifff(x)>0}, there exists $w\in K^*$ such that $u^\top w <0$. Then
        \[
            (u^\top w) u = (uu^\top) w = (vv^\top) w = (v^\top w) v\in K.
        \]
        Hence $u\in -K$.\\
        \ref{Lambda closed}: Since the set of all rank at-most one matrices is a closed set, if $\{u_n\mid n\in \N\}\subseteq K$, is such that $u_nu_n^\top\to A$, then $A=uu^\top$ for some $u\in \R^n$. If $u\in K$, then we are done. If $u\notin K$, then, by Proposition~\ref{prop::xinKifff(x)>0}, there exist $v\in K^*$ such that $u^\top v <0$. Since the map $(A\mapsto Av)\colon S_n\to \R^n$ is continuous, we get
        \[
            K\ni (u_n^\top v)u_n = (u_nu_n^\top)v \to (uu^\top)v= (u^\top v)u.
        \]
        Since $K$ is closed, we get $u\in (-K)$. In any case, we conclude that $u\in K\cup (-K)$, and so $uu^\top\in \Lambda$ proving that $\Lambda$ is closed.
    \end{proof}
     Hence, by Proposition~\ref{prop:Extremals-are-R*}, the set $\Lambda$, identified with the functionals ($A\mapsto u^\top Au$), consists of all the Riesz* functionals on $(S_n, COP(K)).$

\section{Unisigned matrices}\label{Sec::uni-signed}

In this section, we introduce and characterize a special class of matrices which will play an important role in this article. We start with an auxiliary result. 
\begin{lemma}\label{lemma::boundary_lemma}
    Let $A\subseteq X$ be a closed and convex subset of a normed space $(X,\|\cdot\|)$. Let $u\in A$ and $v\notin A$. Consider the map $g\colon [0,1]\to X$, defined as $g(t)=(1-t)u+tv$ for $t\in [0,1]$. Then there exists $c\in [0,1)$ such that $g(t)\in A$ if and only $t\in [0,c].$
\end{lemma}
\begin{proof}
    First, observe that $g$ is continuous. Set $c:=\inf g^{-1}(A^c)$. Note that $1 \in g^{-1}(A^c)$ and $g^{-1}(A^c)$ is an open set. Hence $c<1$. Clearly, $g([0,c))\subseteq A$, and since $g^{-1}(A)$ is closed, we get $g([0,c])\subseteq A$. Notice that, for any $t\in (0,1)$ and $\lambda\in (0, t)$, $g(\lambda)=(1-\frac{\lambda}{t})u+\frac{\lambda}{t}g(t)$. Since $A$ is convex and $u\in A,$ if $g(t)\in A$, then $g((0,t))\subseteq A$. Hence, $g((c,1])\subseteq A^c$. This completes the proof.
\end{proof}
\begin{remark}\label{rem::boundary-lemma}
    Observe that, in Lemma~\ref{lemma::boundary_lemma}, $g$ is a directed line segment from $u$ to $v$. If we change the orientation of $g$, i.e., define $g(t):=tu + (1-t)v$ for $t\in [0,1]$, then there exists $c\in (0,1]$ such that $g(t)\in A$ if and only $t\in [c,1]$.
\end{remark}

\begin{definition}\label{def:unisigned}
    Let $K_1\subseteq \R^n$ and $K_2\subseteq \R^m$ be cones. A matrix $P\in M_n$ is called {\it $(K_1,K_2)$-unisigned} if, for every $x\in K_1$, one has $Px\in K_2\cup (-K_2)$, i.e., $P[K_1]\subseteq K_2\cup (-K_2).$
\end{definition}
When $K_1=K_2=:K$, we simply write $K$-unisigned instead of $(K,K)$-unisigned. The nomenclature is motivated by the case $K_1=\R^n_+$ and $K_2:=\R^m_+$. In this case, $A\in M_{m\times n}$ is $(\R^n_+,\R^m_+)$-unisigned if and only if no column of $A$ has entries with different signs. In the next result, we characterize $(K_1,K_2)$-unisigned matrices under some natural assumptions on cones $K_1$ and $K_2$. Recall that, for cones $K_1\subseteq \R^n$ and $K_2\subseteq \R^m$, $\pi(K_1, K_2):=\{A\in M_{m\times n}\mid A[K_1]\subseteq K_2\}$.

\begin{theorem}\label{thm::prop(U)_char}
    Let $K_1\subseteq \R^n$ be a generating cone and $K_2\subseteq \R^m$ be a closed cone, and let $P\in M_{m\times n}$. Then, $P$ is $(K_1, K_2)$-unisigned if and only if one of the following holds.
    \begin{enumerate}[label=\upshape(\roman*)]
        \item 
            $P\in \pi(K_1, K_2)\cup (-\pi(K_1, K_2))$
        \item 
             $P=uv^\top$ with $u\in K_2\cup (-K_2)$ and $v\in \R^n$.
    \end{enumerate}
\end{theorem}
\begin{proof}
    Sufficiency is straightforward. Suppose that $P$ is $(K_1, K_2)$-unisigned and $P\notin \pi(K_1, K_2)\cup (-\pi(K_1, K_2))$, i.e., we have
    \begin{align}
        P[K_1]&\subseteq K_2\cup (-K_2),\label{1}\\
        P[K_1]&\not\subseteq K_2,\label{2}\\
        P[K_1]&\not\subseteq -K_2.\label{3}
    \end{align}
    From \eqref{1} and \eqref{3}, there exists $x\in K_1$ such that $0\neq Px\in K_2$. Similarly, by \eqref{1} and \eqref{2}, there exists $y\in K_1$ such that $0\neq Py\in -K_2$. Clearly $x$ and $y$ are linearly independent. Since $K_1$ is generating, we have $\R^n=\text{span}(K_1).$ Let $\calB:=\{x,y, u_1,\ldots, u_k\}\subseteq K_1$ be a basis of~$\R^n.$ Let 
    \[
        B_1:=\{z\in \calB\mid Pz\in K_2\} \quad \text{and}\quad \calB_2:=\{z\in \calB\mid Pz\in -K_2\}.
    \]
By \eqref{1}, it is clear that $\calB=\calB_1\cup \calB_2$. Also, $x\in \calB_1$ and $y\in \calB_2$. Suppose that $\text{rank}(P) \geq 2.$ Then there exist at least two linearly independent vectors in $P[\calB]$. We claim that there exist $z_1\in \calB_1$ and $z_2\in \calB_2$ such that $\{Pz_1, Pz_2\}$ is linearly independent. For, if $Px$ and $Py$ are independent, then this holds. Otherwise, there exists $u_i\in \calB$ such that $Pu_i$ and $Px$ are linearly independent. Consequently $Pu_i$ and $Py$ are also linearly independent. If $u_i\in \calB_1$, then choose $z_1:=u_i$ and $z_2:=y$. If $u_i\in \calB_2$, then choose $z_1:=x$ and $z_2:=u_i$.
    
Consider the map $g\colon [0,1]\to \R^n$, given by $g(t):=(1-t)z_1+tz_2$ for $t\in [0,1]$. Since $z_1, z_2\in K_1$, we get $g([0,1])\subseteq K_1$. Thus, by \eqref{1}, for any $t\in [0,1],$
    \begin{equation}\label{eq::3.1}
        P\circ g(t) = (1-t)Pz_1 + t Pz_2\in K_2\cup(-K_2).
    \end{equation}
     Also, $P\circ g(0)=Pz_1\in K_2$ and $P\circ g(1)=Pz_2\in -K_2$. Thus, by  Lemma~\ref{lemma::boundary_lemma} applied to $K_2$, we get $c_1\in [0,1)$ such that
     \[
        P\circ g(t)\in K_2 ~\Longleftrightarrow~ t\in [0,c_1],
     \]
     and Remark~\ref{rem::boundary-lemma} applied to $-K_2$ gives $c_2\in (0,1]$ such that
     \[
        P\circ g(t)\in -K_2 ~\Longleftrightarrow~ t\in [c_2,1].
     \]
     We claim that $[0,c_1]\cap [c_2, 1]=\emptyset$. Suppose on the contrary that $c\in [0,c_1]\cap [c_2, 1]$. Then $P\circ g(c)\in K_2\cap (-K_2)=\{0\}$, implying that $\{Pz_1, Pz_2\}$ is linearly dependent, a contradiction. Thus, $[0,c_1]\cap [c_2, 1]=\emptyset$ and so $c_1<c_2$. This, in turn, implies that $P\circ g(t)\notin K_2\cup (-K_2)$ for all $t\in (c_1,c_2)$, a  contradiction to \eqref{eq::3.1}. Therefore, the rank of $P$ is atmost $1$, and hence $P=uv^\top$ for some $u\in \R^m$ and $v \in\R^n$. Since $P$ is $(K_1, K_2)$-unisigned, we get $u\in K_2\cup(-K_2)$.
\end{proof}

\begin{remark}
   One may define a $(K_1,K_2)$-unisigned operator between  partially ordered vector spaces $(X,K_1)$ and $(Y,K_2)$ in an analogous manner. Consequently, Theorem~\ref{thm::prop(U)_char} admits the following more general formulation. Its proof proceeds along the same lines as that of Theorem~\ref{thm::prop(U)_char}, replacing the finite set $\calB$ with a Hamel basis. We omit it, as it is not needed for the purposes of this article.
    \begin{theorem}
    Let $(X,\|\cdot \|_1, K_1)$ and $(Y,\|\cdot \|_2, K_2)$ be ordered normed spaces with $K_1$ being generating and $K_2$ being closed. A linear map $P\colon X\to Y$ satisfies $P[K_1]\subseteq K_2\cup(-K_2)$ if and only if one of the following holds:
    \begin{enumerate}[label=\upshape(\roman*)]
        \item\label{unisignes:i}
            $P\in \pi(K_1, K_2)\cup (-\pi(K_1,K_2))$ 
        \item 
            there exists $u\in K_2\cup(-K_2)$ and $f\in X'$ such that $P(x)=f(x) u$, for every $x\in X$.
    \end{enumerate}
\end{theorem}
\end{remark}

\section{Riesz* homomorphisms on $S_n$}\label{Riesz*-hom-copositive}

Throughout this section, we assume that $K_1\subseteq \R^n$ and $K_2\subseteq \R^m$ are closed and generating cones. For a matrix $A\in S_n$ and a quadruple of indices $i,j,k,l$, let $A[i,j|k,l]$ denote the submatrix of $A$ corresponding to rows $i,j$ and columns $k,l$. Define the map $\psi_{i,j}^{k,l}\colon S_n\to \R$ as $A\mapsto \det (A[i,j|k,l])$. Recall the following standard result.
\begin{lemma}\label{lem::rank-det}
    For $A\in S_n$, one has $\rank(A)\leq 1$ if and only if $\psi_{i,j}^{k,l}(A)=0$ for all quadruple of indices $i,j,k,l$.    
\end{lemma}

A linear map $T\colon S_n\to S_n$ is referred to as {\it rank one non-increasing}, if $\rank(T(A))\leq 1$ whenever $\rank(A)=~1$.

\begin{theorem}[{\cite[Corollary 3]{Lim01021990}}]\label{thm::lim}
    Let $T\colon S_n\to S_m$ be a rank one non-increasing linear map. Then one of the following holds.
    \begin{enumerate}[label=\upshape(\roman*)]
        \item \label{lim(i)}
            There exists $P\in M_{m\times n}$ and $\lambda\in \{1,-1\}$ such that $T(A)= \lambda P A P^\top$.
        \item \label{lim(ii)}
            There exists a rank one matrix $B\in S_m$ such that $T[S_n]=\text{span}(B)$.
    \end{enumerate}
\end{theorem}

Let $\Lambda_n:=\{uu^\top \mid u\in K_1\}$ and $\Lambda_m:=\{uu^\top \mid u\in K_2\}$. We know, from Subsection~\ref{ex::COP} that $\Lambda_n=Ext(CP(K_1))$ and $\Lambda_m=Ext(CP(K_2))$. In our next result, we characterize linear maps $T\colon S_n\to S_m$ that satisfy $T[\Lambda_n]\subseteq \Lambda_m$. Observe that a linear map $T\colon S_n\to S_m$ is rank one non-increasing if and only if $T[\{u u^\top\mid u\in \R^n\}]\subseteq \{u u^\top\mid u\in \R^m\}$. Hence, the next result may be considered as an extension of Theorem~\ref{thm::lim}.

\begin{theorem}\label{thm::cp-rank-one-non-increasing}
    Let $K_1\subseteq \R^n$ and $K_2\subseteq \R^m$ be cones with $K_1$ being closed and $K_2$ being generating, and let $T\colon S_m\to S_n$ be linear. Then $T[\Lambda_m]\subseteq \Lambda_n$ if and only if $T$ has one of the following forms:
    \begin{enumerate}[label=\upshape(\roman*)]
        \item 
            There exists a $(K_2, K_1)$-unisigned matrix $P\in M_{n\times m}$ such that $T(A)=P A P^\top$ for $A\in S_m$.
        \item 
            There exists $U\in COP(K_2)$ and $y\in K_1$ such that $T(A)=\langle A,U\rangle yy^\top$ for $A\in S_m$.
    \end{enumerate}
\end{theorem}
\begin{proof}
    Sufficiency is easily verified. Let $T[\Lambda_m]\subseteq \Lambda_n$. First, we show that $T$ is rank one non-increasing. Let $A\in S_m$ be a rank one matrix so that there exists $u\in \R^m\setminus \{0\}$ and $\mu\in \{1,-1\}$ such that $A= \mu uu^\top$. Since $\rank(A)=\rank(\alpha A)$ for any nonzero $\alpha \in \R$, without loss of generality, we assume that $\mu=1$. Since $K_2$ is generating, there exist $u_1, u_2\in K_2$ such that $u=u_1-u_2$. For $\alpha\in \R$, we define $u_\alpha:= u_1+ \alpha u_2$ and
    \[
       A_\alpha := u_\alpha u_\alpha^\top = u_1 u_1^\top + \alpha^2 u_2u_2^\top + \alpha (u_1u_2^\top +u_2u_1^\top). 
    \]
     Then, for each $\alpha\geq 0$, we get $u_\alpha\in K_2$ and $A_\alpha\in \Lambda_m$. Notice that $A_{-1}=A$. Fix $1\leq i< j$ and $1\leq k<l$, and consider the map $\Psi\colon \R\to \R$; $\alpha\mapsto \psi_{i,j}^{k,l}(T(A_\alpha))$. It is clear that $\Psi$ is a polynomial map. Since $T[\Lambda_m]\subseteq \Lambda_n$, we get $\rank(T(A_\alpha))\leq 1$ for each  $\alpha \geq 0$. Hence, by Lemma~\ref{lem::rank-det}, it follows that $\Psi(\alpha)=0$ for all $\alpha\geq 0$. Since a nonzero polynomial can have at most finitely many zeros, we get $\Psi=0.$ Hence $\det (T(A_\alpha)[i,j|k,l])=0$ for every $\alpha\in \R$. In particular, $\det (T(A_{-1})[i,j|k,l])=0$. Since $i,j,k,l$ are arbitrary, we get $\rank (T(A))\leq~ 1$. This proves that $T$ is rank one non-increasing.

    Theorem \ref{thm::lim} applies and so $T$ takes one of the two forms given there. Suppose that $T$ is of form \ref{lim(i)} so that $T(A)=\lambda P A P^\top,~A\in S_n$ with $\lambda \in \{-1,1\}$. Notice that, for any $u\in K_1$, $\tr(T(uu^\top))=\lambda \|Pu\|^2$. Since $\tr(U)\geq 0$ for any $U\in \Lambda_m$ and $T(uu^\top)\in \Lambda_m$, we conclude that $\lambda=1$.  If, $Pu\notin K_1\cup(-K_1)$, for some $u\in K_2$, then, by Lemma~\ref{cp-lemma}, $T(uu^\top)=Pu(Pu)^\top\notin \Lambda_n$, a contradiction. Hence $P$ is $(K_2, K_1)$-unisigned.
    
    Suppose $T$ is of form \ref{lim(ii)} so that there exists a rank one matrix $B\in S_n$ and linear functional $f\colon S_m\to\R$ such that
    \[
        T(A)=f(A)B,\quad \forall A\in S_m.
    \]
   Let $y\in \R^n$ {be} such that $B=\pm yy^\top$. Hence $T(A)=\hat{f}(A) yy^\top$ for any $A\in S_m$, where $\hat{f}=\pm f$. Let $U\in S_m$ be such that $\hat{f}(A)=\langle A, U\rangle$ for all $A\in S_m.$ Then, for any $u\in K_2$,
    \[
        T(uu^\top)=\hat{f}(uu^\top) yy^\top=\langle uu^\top, U\rangle yy^\top = (\tr(uu^\top U)) yy^\top= (u^\top U u) yy^\top.
    \]
    Since $T[\Lambda_m] \subseteq \Lambda_n$, it follows that $u^\top U u\geq 0$ and $yy^\top \in \Lambda_n$, showing that $U\in COP(K_2)$ and $y\in K_1\cup (-K_1)$. Since $yy^\top =(-y)(-y)^\top$, without loss of generality, we assume $y\in K_1$. We have shown that there exists $y \in K_1$ and $U\in COP(K_2)$ such that $T(A)=\langle A,U\rangle yy^\top$, completing the proof.
\end{proof}

Next, we prove our third main result, where we characterize all Riesz* homomorphisms on $T\colon S_n\to S_m$, where both the spaces are equipped with the respective cone of $K$-copositive matrices. For $T\colon S_n\to S_m$, the operator $T^*\colon S_m\to S_n$ denotes the (algebraic) dual operator of~$T$.

\begin{theorem}\label{thm::Riesz*hom_on_X_m}
    Let $K_1\subseteq \R^n$ and $K_2\subseteq \R^m$ be closed and generating cones, and $S_n$ and $S_m$ be equipped with the cones $COP(K_1)$ and $COP(K_2)$, respectively. For a linear map $T\colon S_n\to S_m$, the following statements are equivalent:
    \begin{enumerate}[label=\upshape(\roman*)]
        \item \label{label-1}
            $T$ is a Riesz* homomorphism.
        \item \label{label-2}
            $T^*[\Lambda_m]\subseteq \Lambda_n$.
        \item \label{label-3}
            $T$ assumes one of the following two forms:
            \begin{enumerate}
                \item [(a)]
                    $T(A)=P^\top A P$, for some $(K_2,K_1)$-unisigned matrix $P\in M_{n\times m}$. \label{thm::Riesz*hom_on_X_m_eq::A}
                \item [(b)]
                    $T(A)=(y^\top Ay) U$, for some $y\in K_1$ and $U\in COP(K_2)$. \label{thm::Riesz*hom_on_X_m_eq::(iii)}
            \end{enumerate}  
    \end{enumerate}
\end{theorem}
\begin{proof}
    \ref{label-1} $\implies$ \ref{label-2}: Let $T\colon S_n\to S_m$ be a Riesz* homomorphism. Recall from the Subsection~\ref{ex::COP} that 
    \[
        \Lambda_n=Ext((COP(K_1))^*) \quad \text{and}\quad \Lambda_m=Ext((COP(K_2))^*).
    \]
    Using the identification of $(S_n)^*$ with $S_n$, we know by Proposition~\ref{prop:Extremals-are-R*}, that $\Lambda_n$ and $\Lambda_m$ consists of all the Riesz* functionals on $(S^n, COP(K_1))$ and $(S^m, COP(K_2)$, respectively. By Remark \ref{Rem::Riesz*-composition}(ii), if $f\colon S_m\to \R$ is a Riesz* homomorphism, then so is $f\circ T$. Hence, $T^* [\Lambda_m]\subseteq \Lambda_n$.\\
    \ref{label-2} $\implies$ \ref{label-3}: This follows by a straightforward calculation after applying Theorem \ref{thm::cp-rank-one-non-increasing} to $T^*$. \\
    \ref{label-3} $\implies$ \ref{label-1}: Let $F=\{A_1,\ldots, A_k\}\subseteq S_n$ for some $k\in \N$. We first prove that, for any $A\in S_n$ and for any closed and generating cone $K\subseteq \R^n$, one has
    \begin{equation}\label{eq:Ful}
        A\in F^{ul} \quad\Longleftrightarrow\quad\forall x\in K\colon~ \langle Ax,x\rangle\leq \max_{1\leq i\leq k} \langle A_ix,x\rangle.
    \end{equation} 
    Let $A\in S_n$ be such that the inequality in \eqref{eq:Ful} is true. Now, for any $B\in F^u$ and $x\in K$, by definition, we get $\langle Bx,x\rangle\geq \langle A_ix,x\rangle$ for every $1\leq i\leq k$. Hence
    \[
        \langle Bx,x \rangle\geq \max_i \langle A_ix,x\rangle\geq \langle Ax,x\rangle,
    \] 
    showing that $A\in F^{ul}$. Conversely, for $x\in K$, consider the functional $f$ on $S_n$ given by the matrix $xx^\top$, i.e. $f\colon A\mapsto \tr(A(xx^\top))=\langle Ax,x\rangle$. Observe that 
    \[
        f[F]^{ul}=(-\infty, \max_{1\leq i\leq k} \langle A_ix,x\rangle].
    \]
    Since $f$ is a Riesz* functional on $(S_n, COP(K))$ (cf. Subsection~\ref{ex::COP}), the implication ``$\implies$" in \ref{eq:Ful} follows.
    
    Now, let $T$ be of form (a) in \ref{label-3} so that $T(A)=P^\top AP, ~A\in S_n$ with $P\in M_n$ being $(K_2, K_1)$-unisigned. Let $F=\{A_1,\ldots, A_k\}\subseteq S_n$ for some $k\in \N$ and $A\in F^{ul}$. We show that $T(A)\in T[F]^{ul}$. We use \eqref{eq:Ful}. Since $P$ is $(K_2,K_1)$-unisigned, for every $x\in K_2$, we get $Px\in K_1\cup(-K_1)$. Since $A\in F^{ul}$, 
    \[
        \forall x\in K_2\colon \quad  \langle T(A)x,x\rangle = \langle APx,Px\rangle \leq \max_{i}\langle A_iPx,Px\rangle=\max_i \langle T(A_i)x,x\rangle.
    \]
    Again, by using \eqref{eq:Ful}, it follows that $T(A)\in T[F]^{ul}$, showing that $T$ is a Riesz* homomorphism. Let $T$ be of form (b) in \ref{label-3}. Then
    \[
        \forall x\in K_2\colon\quad \langle T(A)x,x\rangle = \langle Ay,y\rangle \langle Ux,x\rangle\leq (\max_i \langle A_iy,y\rangle) \langle Ux,x\rangle=\max_i \langle T(A_i)x,x\rangle,
    \]
    showing that $T(A)\in T[F]^{ul}$.
\end{proof}

\begin{remark}
     Let $(X,K_1)$ and $(Y,K_2)$ be partially ordered vector spaces. In \cite[Proposition 7.1(ii)]{Florian-Valentin-Anke-Janko-Onno}, the authors showed that if $A\subseteq K_2^*$ is {total}\footnote{For a partially ordered vector space $(X,K)$, a subset $A\subseteq K^*$ is called {\it total} if, for every $x\in X$, one has $x\in K$ if and only if $f(x)\geq 0$ for every $f\in A$.} and $T\colon X\to Y$ is a linear map such that $f\circ T$ is a Riesz* functional for every $f\in A$, then $T$ is a Riesz* homomorphism. The implication (ii) $\implies$ (i) in Theorem~\ref{thm::Riesz*hom_on_X_m} can also be obtained using this result.
\end{remark}

\begin{remark}
    The question of whether there exists an invertible Riesz* homomorphism $T$ on a partially ordered vector space $X$ such that $T^{-1}$ is not a Riesz* homomorphism was open until 2018. In \cite{van_Imhoff}, the author gave an example, on the space $P[0,1]$ of all polynomials defined on $[0,1]$, of a bijective Riesz* homomorphism whose inverse is not a Riesz* homomorphism. It is evident from Theorem~\ref{thm::Riesz*hom_on_X_m} that if $T\colon S_n\to S_n$ is a bijective Riesz* homomorphism on $(S_n, COP(K))$, then $T^{-1}$ need not be a Riesz* homomorphism, in general. Indeed, let $P\in M_n$ be such that $P\in \pi(K)$, bijective and $P^{-1}\notin \pi(K)$. Then, by Theorem~\ref{thm::prop(U)_char}, $P$ is $K$-unisigned but $P^{-1}$ is not, and so $T\colon S_n\to S_n$, $A\mapsto P^\top AP$ is a Riesz* homomorphism but $T^{-1}$, given by $T^{-1}(A)=(P^{-1})^\top A P^{-1}$, is not a Riesz* homomorphism. The existence of such $P$ is proved in \cite{raickwade2026invertiblepositivemapsautomorphism}.
\end{remark}

The following corollary, obtained by combining Proposition~\ref{bipositive-is-Riesz*}, Theorem~\ref{thm::prop(U)_char}, and Theorem~\ref{thm::Riesz*hom_on_X_m}, recovers the main results of \cite{shitov2019linearmappingspreservingcopositive} and \cite{GOWDA20133862}.

\begin{corollary}\label{cor::to Riesz*}
    Let $K\subseteq \R^n$ be a closed and generating cone, and $T\colon S_n\to S_n$ be  linear. Then $T\in Aut(COP(K))$ if and only if there exists $P\in Aut(K)$ such that 
    \[
        T(A)=P^\top AP, \quad \forall A\in S_n.
    \]
\end{corollary}

\section{Into preservers of the standard type}\label{into-preservers}

Recall that, in the partially ordered vector space $(S_n,COP(K))$, by definition a linear map $T\colon S_n \to S_n$ is positive if $T[COP(K)]\subseteq COP(K)$, which is equivalent to $T$ being an into preserver of $K$-copositive matrices. Since every Riesz* homomorphism is positive, it is evident from Theorem~\ref{thm::Riesz*hom_on_X_m}(iii) and Theorem~\ref{thm::prop(U)_char} that the conclusion of Claim~\ref{claim1} is not valid. Below, we provide a counter-example.
\begin{example}\label{count-example}
    For $P:=\begin{psmallmatrix}
        1&-1\\
        1&-1
    \end{psmallmatrix}\in M_2$, let $T\colon S^2\to S^2$ be defined by $T(A):=P^\top A P, ~A\in S^2$. Let $A$ be $\R^2_+$-copositive, $(x_1,x_2)^\top=:x\in \R^2_+$ and $e:=(1,1)^\top$. Then 
    \[
        x^\top T(A) x= (Px)^\top A (Px)=(x_1-x_2)^2e^\top Ae \geq 0.
    \] 
    Thus, while $P$ is not entry-wise nonnegative, $T$ is still an into preserver of $\R^2_+$-copositive matrices. Observe that $T$ does not arise from a nonnegative matrix, i.e., there exists no map $\hat{T}\colon S^2\to S^2$, given by $\hat{T}(A)=R^\top A R$ with $R\geq 0$ (entry-wise) such that $T= \hat{T}$. The reason is that while $\hat{T}$ maps nonnegative matrices into nonnegative matrices, we have $T(\begin{psmallmatrix}
        1&0\\
        0&0
\end{psmallmatrix})=\begin{psmallmatrix}
        ~~1&-1\\
        -1&~~1
    \end{psmallmatrix}$. 
\end{example}

\begin{remark}
    In the proof of \cite[Theorem 2.2]{Furtado27072021}, they have showed that $T\colon A\mapsto P^\top AP$ is an into preserver of $\R^n_+$-copositive matrices if and only if $Px\in \R^n_+\cup (-\R^n_+)$ for every $x\in \R^n_+$, i.e., $P$ is $\R^n_+$-unisigned. Then they inferred that this condition implies that $P$ is entry-wise nonnegative. This was the mistake in their proof. We have shown, in Theorem~\ref{thm::prop(U)_char}, that $\R^n_+$-unisigned does not imply entry-wise nonnegative or nonpositive, in general.
\end{remark}

Next, we characterize matrices $P$ {for} which the map $A\mapsto P^\top A P$ preserves $K$-copositivity. As a consequence, we amend the flaw in Claim~\ref{claim1}.
\begin{theorem}\label{thm::copositive_characterizartion}
    Let $K_1\subseteq \R^n$ and $K_2\subseteq \R^m$ be closed and generating cones, and $T\colon S_n\to S_m$, be defined by $T(A)=P^\top A P,~A\in S_n.$ Then $T[COP(K_1)]\subseteq COP(K_2)$ if and only if $P$ is $(K_2,K_1)$-unisigned.
\end{theorem}
\begin{proof}
    Suppose $P$ is not $(K_2,K_1)$-unisigned. Let $x\in K_2$ be such that $Px\notin K_1\cup(-K_1)$. Since $COP(K_1)^*=CP(K_1)$ and $Q:=Px(Px)^\top\notin CP(K_1),$ by Proposition~\ref{prop::xinKifff(x)>0}, there exists $A\in COP(K_1)$ such that $0 > \langle Q, A\rangle = x^\top P^\top APx,$ showing that $P^\top AP$ is not $K_2$-copositive. \\
    Conversely, if $P$ is $(K_2,K_1)$-unisigned, then, for any $x\in K_2$ and $A\in COP(K_1)$, we get 
    \[
        x^\top P^\top AP x = \begin{cases}
            (Px)^\top A (Px)\geq 0, & Px\in K_1,\\
            (-Px)^\top A (-Px)\geq 0, & Px\in -K_1.
        \end{cases}\qedhere
    \]
\end{proof}
\begin{corollary}
Let $T\colon S_n\to S_n$, be defined by $T(A)=P^\top A P,~A\in S_n.$ Then $T$ is an into preserver of $\R^n_+$-copositive matrices if and only if $P$ is $\R^n_+$-unisigned.
\end{corollary}

\textbf{Acknowledgment}
Pavankumar thanks Mr. Florian Boisen and Mr. Sujit Damase for carefully reviewing this article and providing useful suggestions that lead to a better presentation of the results. He also thanks Ministry of Education, Government of India, for providing the Prime Minister's Research Fellowship (PMRF) during his PhD.


\begin{thebibliography}{10}

    \bibitem{cones_and_duality}  C. D. Aliprantis and R. Tourky. Cones and duality, volume 84 of {\it Grad. Stud. Math. Providence, RI: AMS} 2007.

    \bibitem{Florian} Boisen, F. (2026). Mild Riesz* homomorphisms of higher degrees. {\it Positivity}, 30(3):33. Id/No 37.

    \bibitem{Florian-Valentin-Anke-Janko-Onno} Boisen, F., Hölker, V. G., Kalauch, A., Stennder, J., and van Gaans, O. (2024). A generalization of Riesz* homomorphisms on order unit spaces. {\it Quaest. Math.}, 47(9):1887–1911.

    \bibitem{Florian-Anke-Janko-Onno} Boisen, F., Kalauch, A., Stennder, J., and van Gaans, O. (2025). Characterizing Riesz* homomorphisms via interval preserving order adjoints. {\it Positivity}, 29(3):18. Id/No 31.

    \bibitem{The_complete_Riesz_hom_preliminary} Buskes, G. and van Rooij, A. (1993). The vector lattice cover of certain partially ordered groups. {\it J. Aust. Math. Soc., Ser. A}, 54(3):352–367.

    \bibitem{Furtado27072021} S. Furtado, C. R. Johnson, and Y. Zhang. Linear preservers of copositive matrices. {\it Linear and Multilinear Algebra}, 69(10):1779-1788, 2021.


    \bibitem{GOWDA20133862} M. S. Gowda, R. Sznajder, and J. Tao. The automorphism group of a completely positive cone and its lie algebra. {\it Linear Algebra Appl.}, 438(10):3862–3871, 2013.

    \bibitem{copo-survey} J.-B. Hiriart-Urruty and A. Seeger. A variational approach to copositive matrices. {\it SIAM Review}, 52(4):593–629, 2010.

    \bibitem{Anke_book} A. Kalauch and O. van Gaans. {\it Pre-Riesz spaces}, volume 66 of {\it De Gruyter Expo. Math.} Berlin: De Gruyter, 2019.

    \bibitem{Lim01021990} M. H. Lim. Linear mappings on second symmetric product spaces that preserve rank less than or equal to one. {\it Linear and Multilinear Algebra}, 26(3):187–193, 1990.

    \bibitem{raickwade2026invertiblepositivemapsautomorphism}  Raickwade, P. and Sivakumar, K. C. (2026). Invertible positive maps that are not automorphism. Preprint. 

    \bibitem{shitov2019linearmappingspreservingcopositive} Y. Shitov. Linear mappings preserving the copositive cone. {\it Proc. Amer. Math. Soc.},149(8):3173–3176, 2021.

    \bibitem{Haandel+1993} M. van Haandel. {\it Completions in Riesz space theory}. PhD thesis, University of Nijmegen, 1993.    

    \bibitem{van_Imhoff} H. van Imhoff. Riesz* homomorphisms on pre-Riesz spaces consisting of continuous functions. {\it Positivity}, 22(2):425–447, 2018.
\end{thebibliography}
\end{document}